\documentclass[11pt]{article}
\usepackage[T1]{fontenc}
\usepackage{lmodern}
\usepackage{amsmath,amssymb,amsthm}
\usepackage[margin=1.2in]{geometry}
\usepackage{microtype}
\usepackage{needspace}
\DeclareUnicodeCharacter{2124}{\ensuremath{\mathbb{Z}}}
\DeclareUnicodeCharacter{211D}{\ensuremath{\mathbb{R}}}
\DeclareUnicodeCharacter{2115}{\ensuremath{\mathbb{N}}}
\DeclareUnicodeCharacter{03C6}{\ensuremath{\varphi}}
\DeclareUnicodeCharacter{03C3}{\ensuremath{\sigma}}
\DeclareUnicodeCharacter{03B5}{\ensuremath{\varepsilon}}
\DeclareUnicodeCharacter{2200}{\ensuremath{\forall}}
\DeclareUnicodeCharacter{2203}{\ensuremath{\exists}}
\DeclareUnicodeCharacter{2208}{\ensuremath{\in}}
\DeclareUnicodeCharacter{2264}{\ensuremath{\le}}
\DeclareUnicodeCharacter{2227}{\ensuremath{\wedge}}
\DeclareUnicodeCharacter{2260}{\ensuremath{\ne}}
\DeclareUnicodeCharacter{2192}{\ensuremath{\to}}
\DeclareUnicodeCharacter{2218}{\ensuremath{\circ}}
\DeclareUnicodeCharacter{2286}{\ensuremath{\subseteq}}
\DeclareUnicodeCharacter{2191}{\ensuremath{\uparrow}}
\DeclareUnicodeCharacter{2243}{\ensuremath{\simeq}}
\DeclareUnicodeCharacter{1D62}{\ensuremath{{}_i}}
\DeclareUnicodeCharacter{2080}{\ensuremath{{}_0}}
\theoremstyle{plain}
\newtheorem{theorem}{Theorem}[section]
\newtheorem{lemma}[theorem]{Lemma}
\newtheorem{proposition}[theorem]{Proposition}
\newtheorem{corollary}[theorem]{Corollary}
\theoremstyle{definition}
\newtheorem{definition}[theorem]{Definition}
\newtheorem{question}[theorem]{Question}
\theoremstyle{remark}
\newtheorem{remark}[theorem]{Remark}
\newcommand{\Z}{\mathbb{Z}}
\newcommand{\N}{\mathbb{N}}
\newcommand{\R}{\mathbb{R}}
\usepackage[hidelinks]{hyperref}
\hypersetup{
  pdftitle={Bijective Non-Contractive Maps of Hereditarily
    Precompact Quasi-Metric Spaces Need Not Be Isometries},
  pdfauthor={Thomas Duzer},
  pdfkeywords={quasi-metric space, expand-contract plasticity,
    non-contractive map, isometry, hereditary precompactness,
    formal verification, Lean}
}
\title{Bijective Non-Contractive Maps of Hereditarily\\
Precompact Quasi-Metric Spaces\\ Need Not Be Isometries}
\author{Thomas Duzer\thanks{thomas.duzer@gmail.com.
  ORCID: \href{https://orcid.org/0009-0007-2147-081X}{0009-0007-2147-081X}.}}
\date{}
\begin{document}
\maketitle
\begin{abstract}
Zavarzina asked whether every bijective non-contractive self-map of
a precompact quasi-metric space is an isometry. We show that it need
not be, even when the space is hereditarily precompact. On the
integers, take the distance of an irrational rotation and add a fixed
toll $c > 0$ to every step from a nonnegative integer down to a
smaller one. The resulting space, the toll spiral, is a hereditarily
precompact quasi-metric space, and the shift $n \mapsto n + 1$ is a
bijective non-contractive map that is not an isometry. The
construction works over any metric space $Y$: for every bijective
isometry $\sigma$ of $Y$ and every point whose $\sigma$-orbit is
aperiodic and totally bounded, adding the same toll along the orbit
gives such a counterexample. Instances include the irrational
rotation and the $p$-adic odometer. The conjugate of the toll spiral
is not even precompact, so the hypothesis on the conjugate in
Zavarzina's positive result cannot be dropped, even for bijections.
The toll spiral is also not expand-contract plastic, since the
inverse shift is a non-expansive bijection that is not an isometry.
For an arbitrary bijective non-contractive map $F$, the distance
from $F^k x$ to $F^k y$ has a limit as $k \to -\infty$; $F$ is an
isometry for this limit, and the excess $D_0$ of the distance over it
is the defect. On a hereditarily precompact space, we prove a
monochromy theorem: if a pair $(a,b)$ has positive defect $\eta$,
then there are natural numbers $t_1 < t_2 < \cdots$ such that either
$D_0(a, F^{-\Delta}a) > \eta/4$ for every difference
$\Delta = t_j - t_i$ with $i < j$, or $D_0(F^{-\Delta}b, b) > \eta/4$
for every such $\Delta$. Three boundary results delimit the remaining
question; the open case, that of separated spaces compact in one of
their topologies, has since been settled by the author
(arXiv:2610.14301; see the note at the end of
Section~\ref{sec:questions}). The construction, in
general and for the circle, is machine-checked in Lean~4 using
Mathlib, with no unproved steps, under a pinned toolchain.
\end{abstract}


\section{Introduction}\label{sec:intro}

A self-map $F$ of a quasi-metric space $(X,d)$ is
\emph{non-contractive} if $d(F(x),F(y)) \ge d(x,y)$ for all $x,y$,
and \emph{non-expansive} if the reverse inequality holds throughout.
For metric spaces there is a classical rigidity phenomenon, which
goes back to Freudenthal and Hurewicz \cite{FH1936} and is due, in
the form used here, to Naimpally, Piotrowski, and Wingler
\cite{NPW2006}: on a totally bounded metric space, no self-map can
expand some pair of points without contracting another.
Consequently, every precompact metric space is \emph{expand-contract
plastic}: every non-expansive bijection of it onto itself is an
isometry. Whether the unit ball of every Banach space is
expand-contract plastic is a well-studied open problem
\cite{CKOW2016,KZ2018,Zav2018,AKZ2019}, and the property and its
relatives remain under active study \cite{KZ2024,LZ2024,BMZ2026}.

Zavarzina \cite{Zav2023} studied the same question for quasi-metric
spaces, where the rigidity breaks down. In her Theorem~2, she gives a
hereditarily precompact quasi-metric space with a non-contractive
self-map that is not an isometry. The points of her space form a
one-sided sequence on the unit circle. The distance is the chord
distance between points of the same parity taken in increasing
order, and one of two constant values for all other pairs. The map
moves each point two places along the sequence. It is therefore not
surjective, and her construction depends on this. Plasticity for
maps that need not be bijective is studied in \cite{KZ2024} as
\emph{strong plasticity}. In her
Theorem~3, Zavarzina shows that rigidity returns under a two-sided
hypothesis: if both $(X,d)$ and its conjugate $(X,d^{-1})$ are
hereditarily precompact, then any map that expands one pair of
points must contract another. Between the two results, she writes
that she ``tried to find an analogue of [the] example \ldots{} with a
bijective function $F$, but surprisingly failed,'' and asks:

\begin{quote}
\emph{Is it true that for every precompact quasi-metric space $(X,d)$
every bijective non-contractive mapping $F\colon X \to X$ is an
isometry?} \cite{Zav2023}
\end{quote}

We show that the answer is no, even for hereditarily precompact
spaces. Fix an irrational number $\theta$ and a constant $c > 0$. On $X = \Z$, let
\[
d(n,m) \;=\; \|(n-m)\theta\| \;+\;
\begin{cases}
c & \text{if } m < n \text{ and } n \ge 0,\\
0 & \text{otherwise},
\end{cases}
\]
where $\|t\|$ is the distance from $t$ to the nearest integer, so
that the first term is the arc distance between $n\theta$ and
$m\theta$ on the circle $\R/\Z$. The second term, the
\emph{toll}, adds $c$ to the distance from a nonnegative integer to
any smaller integer. We call this space the \emph{toll spiral}. The shift $S(n) = n + 1$ is a bijective
non-contractive map that is not an isometry. For instance,
$d(0,-1) = d(-1,-2) + c$: the two pairs have the same arc distance,
but only the first is charged the toll.

The toll spiral is hereditarily precompact: every sequence of
integers has a constant, negative, or increasing subsequence, and along such a
subsequence the toll is never charged from a term to a later one. Zavarzina's example
lives on a one-sided sequence, where a shift cannot be surjective.
On $\Z$ the shift is a bijection, but every point now has an
infinite backward orbit, which eventually stays among the negative
integers. Among the negative integers no toll is
charged, so the distance there is symmetric; reaching them from a
nonnegative integer always costs the toll $c$.

We prove the theorem for a general version of this construction
(Theorem~\ref{thm:family}). Start
from any metric space $(Y,\rho)$, a bijective isometry $\sigma$ of
$Y$, and a point $y_0$ whose orbit $\{\sigma^n y_0 : n \in \Z\}$ is
aperiodic and totally bounded, and define a distance on $\Z$ by
$d(n,m) = \rho(\sigma^n y_0, \sigma^m y_0) + \tau(n,m)$, where
$\tau$ is the same toll as above. The result is again a hereditarily
precompact quasi-metric space on which the shift is a bijective
non-contractive map but not an isometry. The space $Y$ need not be
compact; only the orbit must be totally bounded. Besides the
irrational rotation, examples include the $p$-adic odometer, which
is non-Archimedean (Corollary~\ref{cor:padic}).

Section~\ref{sec:conjugate} shows that the conjugate of the toll
spiral is not even precompact, since the distance from a nonnegative
integer to any smaller one is at least $c$. Thus the hypothesis on
the conjugate in Zavarzina's Theorem~3 cannot be dropped, even for
hereditarily precompact spaces and bijective maps.
Section~\ref{sec:conjugate} also shows that the toll spiral is not
expand-contract plastic: the inverse shift $S^{-1}$ is a
non-expansive bijection that is not an isometry.

\paragraph{Terminology.}
We think of $d(x,y)$ as the cost of a trip from $x$ to $y$. A trip
back may cost a different amount, and the triangle inequality says
that a detour is never cheaper than the direct route. This is the
language of transport costs, going back to Monge and Kantorovich,
and of Lawvere's generalized metric spaces, which are categories
enriched in costs \cite{Lawvere1973}. The \emph{toll} is a
surcharge on some trips, and a \emph{fee} is the cost of a trip
between a point and an earlier point of its orbit,
$d(x, F^{-\Delta}x)$ or $d(F^{-\Delta}y, y)$.
Section~\ref{sec:anatomy} bounds the fees from below. It also
introduces the \emph{register}, the part of the distance that the
map preserves and that remains after all surcharges have lapsed.

\paragraph{Verification.}
The main theorem (Theorem~\ref{thm:family}), its circle instance
(Corollary~\ref{cor:circle}), and the results of
Section~\ref{sec:construction} on which they rest are checked in
Lean~4 using Mathlib. Hereditary precompactness is checked
both in its sequential form, whose equivalence with the definition
is cited from \cite[Theorem~3]{KMRV1993}, and directly from the
definition. The formal proofs have no unproved steps and use only
the standard axioms of Lean (Section~\ref{sec:mech}). The $p$-adic
instance (Corollary~\ref{cor:padic}), the exact increase by $c$ in
Proposition~\ref{prop:noncontr}, the necessity of (A2) noted after
Theorem~\ref{thm:family}, and Sections~\ref{sec:conjugate},
\ref{sec:anatomy} and~\ref{sec:questions} are proved on paper; apart from one
informal remark in Section~\ref{sec:anatomy}, which is marked as
such, all these proofs are complete.

\paragraph{Organization.}
Section~\ref{sec:prelim} recalls the definitions. The toll spiral
is constructed in Section~\ref{sec:construction}, where the main
theorem is proved, and its conjugate is studied in
Section~\ref{sec:conjugate}. Section~\ref{sec:anatomy} turns to
bijective non-contractive maps in general, and
Section~\ref{sec:mech} describes the formal verification. Section~\ref{sec:questions} discusses open questions
and proves three results that delimit the first of them; it ends with a note on
the questions answered since the first version.

\section{Preliminaries}\label{sec:prelim}

Following \cite{Zav2023,Cobzas2013}, a \emph{quasi-metric} on a set
$X$ is a function $d\colon X \times X \to [0,\infty)$ such that
(a) $d(x,x) = 0$; (b) $d(x,z) \le d(x,y) + d(y,z)$; and
(c) $d(x,y) = 0 \Rightarrow x = y$. Without (c), $d$ is a
\emph{quasi-pseudo-metric}. The \emph{conjugate} of $d$ is the
quasi-metric $d^{-1}(x,y) = d(y,x)$.

We take the following definitions from \cite{Zav2023}. The ball of
radius $\varepsilon$ at $x$ is
$B(x,\varepsilon) = \{y \in X : d(x,y) < \varepsilon\}$, the
\emph{forward} ball; all balls in this paper are forward balls. The
space $(X,d)$ is \emph{precompact} if for every $\varepsilon > 0$
there is a finite $E \subseteq X$ with
$X \subseteq \bigcup_{x \in E} B(x,\varepsilon)$, and
\emph{hereditarily precompact} if every subspace is precompact. In
the asymmetric setting this is a real strengthening, since a
subspace of a precompact space need not be precompact
\cite{Zav2023}. The net form of hereditary precompactness in the
Lean file (\texttt{hereditarilyPrecompact\_net},
Section~\ref{sec:mech}) states exactly this definition.

\begin{samepage}
A sequence $(x_n)$ is \emph{left K-Cauchy} if for every
$\varepsilon > 0$ there is $N$ with $d(x_n, x_m) < \varepsilon$ for
all $m \ge n \ge N$. Throughout, we use the sequential
characterization of K\"unzi, Mr\v{s}evi\'c, Reilly, and
Vamanamurthy \cite[Theorem~3]{KMRV1993}, as Zavarzina does
\cite{Zav2023}:
\nopagebreak

\begin{quote}
$(X,d)$ is hereditarily precompact if and only if every sequence in
$X$ has a left K-Cauchy subsequence.
\end{quote}
\end{samepage}

A self-map $F$ is \emph{non-contractive} if
$d(F(x),F(y)) \ge d(x,y)$ for all $x,y$, \emph{non-expansive} if
$d(F(x),F(y)) \le d(x,y)$ for all $x,y$, and an \emph{isometry} if
equality always holds. A bijection $F$ of $X$ onto itself is
non-contractive if and only if $F^{-1}$ is non-expansive, so
expand-contract plasticity can equally be defined with
non-contractive bijections, as in Zavarzina's question. If $F$ is a non-contractive
bijection, the distances $\lambda_k(x,y) = d(F^k x, F^k y)$,
$k \in \Z$, satisfy $\lambda_k \le \lambda_{k+1}$, by
non-contractivity at the pair $(F^k x, F^k y)$. We call this the
\emph{ladder relation}. Thus $k \mapsto \lambda_k(x,y)$ is
nondecreasing on all of $\Z$.

\begin{remark}[Convention independence]\label{rem:convention}
Some authors orient balls and K-Cauchy sequences the other way,
and both conventions occur even in the work of K\"unzi and his
coauthors. The choice does not matter here, because conjugation
preserves bijectivity and non-contractivity:
$d^{-1}(Sn, Sm) = d(Sm, Sn) \ge d(m,n) = d^{-1}(n,m)$. So under the
other convention the conjugate space $(\Z, d^{-1})$, with the same
shift, is the counterexample, and the roles of
Sections~\ref{sec:construction} and~\ref{sec:conjugate} are
exchanged. We use the conventions of \cite{Zav2023}, including
its form of the characterization \cite[Theorem~3]{KMRV1993}.
\end{remark}

\section{The construction}\label{sec:construction}

Let $(Y,\rho)$ be a metric space with a bijective isometry $\sigma$,
and let $y_0 \in Y$. We write $\varphi(n) = \sigma^n y_0$ for
$n \in \Z$.

\begin{definition}[Toll spiral]\label{def:toll}
For $c > 0$, the \emph{toll spiral} over $(Y,\rho,\sigma,y_0,c)$ is
$\Z$ with the distance
\[
d(n,m) \;=\; \rho(\varphi(n), \varphi(m)) \;+\; \tau(n,m),
\]
where the \emph{toll} $\tau$ is given by
\[
\tau(n,m) \;=\;
\begin{cases}
c & \text{if } m < n \text{ and } n \ge 0,\\
0 & \text{otherwise}.
\end{cases}
\]
\end{definition}

The first term is the \emph{base part} of $d(n,m)$. A pair $(n,m)$
with $m < n$ is called a \emph{descent} with \emph{source} $n$, and an
integer is \emph{recent} if it is nonnegative and \emph{deep} if it
is negative. In these terms, $d$ coincides with the symmetric base
distance on the deep integers, and $d(n,m) \ge c$ for every descent
$(n,m)$ with recent source.

\begin{samepage}
Besides $\sigma$ being an isometry, we use two conditions on $y_0$:
\nopagebreak
\begin{itemize}
\item[(A1)] \emph{Aperiodicity:} $\sigma^k y_0 \ne y_0$ for every
$k \ne 0$; equivalently, $\varphi$ is injective.
\item[(A2)] \emph{Precompact orbit:} the orbit
$\varphi(\Z) = \{\sigma^n y_0 : n \in \Z\}$ is totally bounded
in $(Y,\rho)$.
\end{itemize}
\end{samepage}

Since $\sigma$ is an isometry, the base part is invariant under the
shift $S(n) = n+1$. We call this property of $\varphi$
\emph{equivariance}:
\begin{equation}\label{eq:equiv}
\rho(\varphi(n+1), \varphi(m+1)) = \rho(\varphi(n), \varphi(m))
\qquad (n, m \in \Z).
\end{equation}

\subsection{Quasi-metric axioms}

\begin{lemma}[Indicator subadditivity]\label{lem:tollsub}
$\tau(n,k) \le \tau(n,m) + \tau(m,k)$ for all $n,m,k \in \Z$.
\end{lemma}

\begin{proof}
Suppose that $\tau(n,k) = c$, so that $k < n$ and $n \ge 0$;
otherwise there is nothing to prove. If $m < n$, then
$\tau(n,m) = c$. If $m \ge n$, then $m \ge 0$ and $k < m$, so
$\tau(m,k) = c$. 
\end{proof}

\begin{lemma}[Indicator monotonicity]\label{lem:tollmono}
$\tau(n,m) \le \tau(n+1, m+1)$ for all $n,m \in \Z$. Moreover, if
$m < n$, then $\tau(n+t, m+t) = 0$ for $t < -n$ and
$\tau(n+t, m+t) = c$ for $t \ge -n$. So, as a descent is shifted forward, it
is charged from the moment its source becomes recent.
\end{lemma}

\begin{proof}
The shift takes a descent with recent source $n$ to a descent with
recent source $n+1$, which gives the first claim. For the second,
$(n+t, m+t)$ is a descent for every $t$, and its source $n+t$ is
recent exactly when $t \ge -n$.
\end{proof}

\begin{proposition}[Quasi-metric]\label{prop:qm}
The distance $d$ is a quasi-pseudo-metric on $\Z$, even without
(A1) and (A2). Under (A1) it is a quasi-metric, and under (A2) it is
bounded:
$\sup_{n,m} d(n,m) \le \operatorname{diam}_\rho(\varphi(\Z)) + c
< \infty$.
\end{proposition}

\begin{proof}
The function $(n,m) \mapsto \rho(\varphi(n),\varphi(m))$ is a
pseudo-metric on $\Z$, and $\tau$ vanishes on the diagonal and
satisfies the triangle inequality by Lemma~\ref{lem:tollsub}. Hence
their sum $d$ is a quasi-pseudo-metric. If
$d(n,m) = 0$, then $\varphi(n) = \varphi(m)$, so $n = m$ by (A1).
Finally, $d \le \operatorname{diam}_\rho(\varphi(\Z)) + c$, and the
diameter is finite by (A2).
\end{proof}

\subsection{The shift is non-contractive but not an isometry}

\begin{proposition}[Non-contractivity and strictness]
\label{prop:noncontr}
The shift $S(n) = n + 1$ is a bijective non-contractive self-map of
$(\Z,d)$, and it is not an isometry: for every pair $(n,m)$ with
$m < n < 0$, after $|n|$ applications of $S$,
\[
d\bigl(S^{|n|} n,\, S^{|n|} m\bigr) \;=\; d(n,m) + c ,
\]
and in particular $d(0,-1) = d(-1,-2) + c$.
\end{proposition}

\begin{proof}
By \eqref{eq:equiv}, $d(n,m)$ and $d(n+1,m+1)$ have the same base
part, so their difference is $\tau(n+1,m+1) - \tau(n,m)$, which is
nonnegative by Lemma~\ref{lem:tollmono}. Thus $S$ is
non-contractive. Now let $m < n < 0$. Then $S^{|n|}$ moves the pair
$(n, m)$ to $(0, m - n)$. By \eqref{eq:equiv}, the two pairs have
the same base part. The toll of $(n, m)$ is $0$, because $n$ is
negative, and the toll of $(0, m - n)$ is $c$, because $0$ is recent
and $m - n < 0$. Hence the distance increases by exactly $c$. For
$(n,m) = (-1,-2)$, this gives $d(0,-1) = d(-1,-2) + c$.
\end{proof}

\begin{remark}[The ladder of a pair]
The ladder of a pair $(n,m)$ is the function
$t \mapsto d(n+t, m+t)$. If $m \ge n$, the toll of $(n+t, m+t)$ is $0$ for every
$t$, so the ladder is constant by \eqref{eq:equiv}. If
$m < n$, the ladder equals $\rho(\varphi(n),\varphi(m))$ for
$t < -n$ and $\rho(\varphi(n),\varphi(m)) + c$ for $t \ge -n$, by
\eqref{eq:equiv} and Lemma~\ref{lem:tollmono}. Hence every ladder
is nondecreasing, and it increases at most once, by $c$.
\end{remark}

\subsection{Hereditary precompactness}

The following extraction lemma is classical, and we include its
proof for completeness.

\begin{lemma}[Precompact-orbit extraction]\label{lem:extract}
Let $(Y,\rho)$ be a metric space and let $(y_k)_{k \in \N}$ be a
sequence whose range is totally bounded. Then $(y_k)$ has a
$\rho$-Cauchy subsequence.
\end{lemma}

\begin{proof}
We choose infinite sets
$\N \supseteq I_1 \supseteq I_2 \supseteq \cdots$ by induction.
Given $I_{j-1}$, with $I_0 = \N$, the points $y_k$ with
$k \in I_{j-1}$ lie in the totally bounded range, so finitely many
balls of radius $2^{-j}$ cover them, and one of these balls contains
$y_k$ for infinitely many $k \in I_{j-1}$. We let $I_j$ be the set of
these $k$. Now pick $k_1 < k_2 < \cdots$ with $k_j \in I_j$. For
$i, i' \ge j$, the indices $k_i$ and $k_{i'}$ both lie in $I_j$, so
$\rho(y_{k_i}, y_{k_{i'}}) < 2^{1-j}$. Thus $(y_{k_j})$ is
Cauchy.
\end{proof}

\begin{theorem}[Hereditary precompactness]\label{thm:hered}
Under (A2), $(\Z,d)$ is hereditarily precompact.
\end{theorem}

We give two proofs. The first works from the definition, the second
from the sequential characterization \cite[Theorem~3]{KMRV1993}.
Both are mechanized, the second up to that characterization
(Section~\ref{sec:mech}).

\begin{proof}[Direct proof]
Let $W \subseteq \Z$ and $\varepsilon > 0$; we must find a finite
$\varepsilon$-net for $W$ inside $W$. Let $W^-$ and $W^+$ be the
sets of negative and of nonnegative elements of $W$. Since no toll
is charged from a negative integer, $d$ coincides on $W^-$ with the
pseudometric $\rho(\varphi(\cdot),\varphi(\cdot))$, for which $W^-$
is totally bounded by (A2). Hence $W^-$ has a finite
$\varepsilon$-net inside itself. To handle $W^+$, cover $\varphi(W^+)$ by
finitely many $\rho$-balls of radius $\varepsilon/2$. For each of
these balls $U$ that meets $\varphi(W^+)$, let $a_U$ be the
\emph{smallest} element $a$ of $W^+$ such that $\varphi(a) \in U$.
Any $y \in W^+$ has $\varphi(y)$ in some such $U$, and then
$a_U \le y$, so that
$d(a_U,y) = \rho(\varphi(a_U),\varphi(y)) < \varepsilon$. Together
with the net for $W^-$, the points $a_U$ form a finite
$\varepsilon$-net for $W$ inside $W$.
\end{proof}

\begin{proof}[Sequential proof]
By \cite[Theorem~3]{KMRV1993}, it suffices to show that every
sequence in $(\Z,d)$ has a left K-Cauchy subsequence. Call a
sequence $(w_k)$ in $\Z$ \emph{toll-free in order} if
$\tau(w_k, w_l) = 0$ whenever $k \le l$.

\emph{Claim 1: every sequence $(u_k)$ in $\Z$ has a subsequence
that is toll-free in order.} A constant subsequence is toll-free in
order, and so is any subsequence of negative terms, because no toll
is charged from a negative integer. If neither kind of
subsequence exists, each value occurs finitely often and there are
finitely many negative terms, so $u_k \to +\infty$. Then $(u_k)$ has
a strictly increasing subsequence, which is toll-free in order.

\emph{Claim 2: a sequence $(w_k)$ that is toll-free in order has a
left K-Cauchy subsequence.} By (A2) and Lemma~\ref{lem:extract},
there is a subsequence $(w_{g(k)})$ for which
$(\varphi(w_{g(k)}))$ is $\rho$-Cauchy. Like $(w_k)$, it is
toll-free in order, so, for $k \le l$,
\[
d(w_{g(k)}, w_{g(l)})
= \rho\bigl(\varphi(w_{g(k)}), \varphi(w_{g(l)})\bigr).
\]
The right-hand side tends to $0$ as $k \to \infty$, uniformly in
$l \ge k$, because $(\varphi(w_{g(k)}))$ is $\rho$-Cauchy. Hence
$(w_{g(k)})$ is left K-Cauchy.
\end{proof}

\begin{remark}[Where the toll hides]\label{rem:blindspot}
For $0 \le m < n$ we have $d(n,m) \ge c$: going backwards, the
nonnegative integers are uniformly separated. This separation
does not affect the sequential criterion, which only involves the
distances $d(w_k, w_l)$ with $k \le l$: every sequence of
nonnegative integers has a nondecreasing subsequence, and for such a
subsequence these distances carry no toll. In the conjugate space the separation is visible,
and Section~\ref{sec:conjugate} uses it to show that the conjugate
is not precompact.
\end{remark}

\subsection{The main theorem and the circle instance}

\begin{theorem}[The counterexample family]\label{thm:family}
Let $(Y,\rho)$ be a metric space, $\sigma$ a bijective isometry of
$Y$, $y_0 \in Y$ satisfying (A1) and (A2), and $c > 0$. Then the
toll spiral $(\Z, d)$ over $(Y,\rho,\sigma,y_0,c)$ is a bounded,
hereditarily precompact quasi-metric space, and the shift
$S(n) = n+1$ is a bijective non-contractive self-map of it which is
not an isometry.
\end{theorem}

\begin{proof}
This follows from Propositions~\ref{prop:qm}
and~\ref{prop:noncontr} and Theorem~\ref{thm:hered}.
\end{proof}

Condition (A2) is also necessary. Suppose that the toll spiral is
precompact, and let $\varepsilon > 0$. There is a finite set
$E_\varepsilon \subseteq \Z$ such that every $n \in \Z$ satisfies
$d(x,n) < \varepsilon$ for some $x \in E_\varepsilon$, and then
$\rho(\varphi(x),\varphi(n)) < \varepsilon$. So finitely many
$\rho$-balls of radius $\varepsilon$ cover the orbit, which is (A2).
Thus, for a toll spiral, precompactness, hereditary precompactness
and (A2) are equivalent.

\begin{corollary}[The toll spiral over the circle]\label{cor:circle}
Let $\theta$ be irrational and $c > 0$. Then the distance on
$X = \Z$ given by
\[
d(n,m) \;=\; \|(n-m)\theta\| \;+\; \tau(n,m),
\]
where $\|t\|$ is the distance from $t$ to $\Z$, is a bounded,
hereditarily precompact quasi-metric, and the shift is a bijective
non-contractive map that is not an isometry. In particular, the
answer to Zavarzina's question is negative, even for hereditarily
precompact spaces.
\end{corollary}

\begin{proof}
Let $\rho$ be the quotient metric on $\R/\Z$, so that
$\rho(x+\Z, y+\Z) = \|x-y\|$, let $\sigma(x) = x + \theta$, and take
$y_0 = 0$. Then $\varphi(n) = n\theta + \Z$, so the base part is
$\|(n-m)\theta\|$ and the toll spiral over $(\R/\Z,\rho,\sigma,0,c)$
is the space of the statement. The rotation $\sigma$ is a bijective
isometry of $\rho$. Condition (A1) holds because $k\theta \notin \Z$
for $k \ne 0$, $\theta$ being irrational, and (A2) holds because
$\R/\Z$ is compact. Now apply Theorem~\ref{thm:family}.
\end{proof}

The chord distance $|e^{2\pi i n\theta} - e^{2\pi i m\theta}|$
used in \cite{Zav2023} would serve as well as the arc metric. The
name \emph{toll spiral} comes from this instance: the points
$(n, n\theta)$ form an endless helix over the circle.

\begin{corollary}[A non-Archimedean instance: the $p$-adic
odometer]
\label{cor:padic}
Let $p$ be a prime and $c > 0$. Then the distance on $X = \Z$
given by
\[
d(n,m) \;=\; |n-m|_p \;+\; \tau(n,m),
\]
where $|\cdot|_p$ is the $p$-adic absolute value, is a bounded,
hereditarily precompact quasi-metric, and the shift is a bijective
non-contractive map that is not an isometry.
\end{corollary}

\begin{proof}
Let $\rho(x,y) = |x-y|_p$ be the $p$-adic metric on $\Z_p$, let
$\sigma(x) = x+1$, and take $y_0 = 0$. Then
$\varphi(n) = \sigma^n(0) = n$, so the base part of the toll spiral
over $(\Z_p,\rho,\sigma,0,c)$ is $|n-m|_p$, and its distance is the
$d$ of the statement. Since $\rho$ is translation-invariant, $\sigma$
is an isometry, and it is bijective with inverse $x \mapsto x-1$.
Condition (A1) holds because $\varphi$ is injective, and (A2) holds
because $\Z_p$ is compact. Now apply Theorem~\ref{thm:family}.
\end{proof}

Both examples are translations of compact monothetic groups, $\R/\Z$
and $\Z_p$, although the construction itself needs neither a group
structure nor compactness.

\section{The conjugate space, sharpness, and plasticity}
\label{sec:conjugate}

\begin{proposition}[The conjugate is not precompact]
\label{prop:conjugate}
Let $(\Z,d)$ be any toll spiral. Then $(\Z, d^{-1})$ is not
precompact: for $\varepsilon \le c$, no finite family of
$d^{-1}$-balls covers $\Z$.
\end{proposition}

\begin{proof}
Let $A \subseteq \Z$ be finite and $\varepsilon \le c$, and choose an
integer $x > \max(A \cup \{0\})$. For every $a \in A$ we have $a < x$
and $x \ge 0$, so
\[
d^{-1}(a, x) = d(x, a) = \rho(\varphi(x),\varphi(a)) + c \ \ge\ c
\ \ge\ \varepsilon .
\]
Thus $x$ lies in no ball $B_{d^{-1}}(a,\varepsilon)$ with
$a \in A$.
\end{proof}

Zavarzina's Theorem~3 \cite{Zav2023} states that if both $(X,d)$ and
$(X,d^{-1})$ are hereditarily precompact, then every function that
expands some pair of points contracts another.

\begin{corollary}[Sharpness of the two-sided theorem]
\label{cor:sharp}
The hypothesis on $(X,d^{-1})$ in Zavarzina's Theorem~3 cannot be
dropped, even for bijections. The toll spiral is hereditarily
precompact, its conjugate is not even precompact, and its shift is a
bijection that expands infinitely many pairs and contracts none.
\end{corollary}

\begin{remark}
For maps that are not bijective, Zavarzina's Theorem~2 already shows
that the hypothesis on the conjugate cannot be dropped: her space is
hereditarily precompact, and her map is non-contractive but not an
isometry. Corollary~\ref{cor:sharp} extends this to bijections.
Remark~\ref{rem:blindspot} and Proposition~\ref{prop:conjugate} both
come from the inequality $d(n,m) \ge c$ for $m < n$ with $n \ge 0$. This
inequality does not stop sequences from having left K-Cauchy
subsequences, but it keeps the conjugate from being precompact.
\end{remark}

\begin{corollary}[Failure of expand-contract plasticity]
\label{cor:plastic}
The toll spiral is a bounded, hereditarily precompact quasi-metric
space that is not expand-contract plastic: the inverse shift
$S^{-1}$ is a non-expansive bijection of $(\Z,d)$ onto itself that
is not an isometry.
\end{corollary}

\begin{proof}
Since $S$ is non-contractive (Proposition~\ref{prop:noncontr}),
$d(S^{-1}n, S^{-1}m) = d(n-1, m-1) \le d(n,m)$ for all $n,m$.
Equality fails at $(0,-1)$, because $d(-1,-2) = d(0,-1) - c$ by
Proposition~\ref{prop:noncontr}. The strict inequality
$d(-1,-2) < d(0,-1)$, which is machine-checked, would already
suffice.
\end{proof}

\section{Anatomy: the structure behind the counterexample}
\label{sec:anatomy}

In this section, $F\colon X \to X$ is a bijective non-contractive
map of a quasi-pseudo-metric space $(X,d)$, and separation is not
assumed unless stated. We introduce the invariant limit register and
prove the covering inequality and the monochromy theorem. The
section ends with an informal remark on the design of the
construction and a proposition on conservative spaces
(Proposition~\ref{prop:conserve}). Except for that remark,
everything is proved in full. Nothing in
Sections~\ref{sec:construction}--\ref{sec:conjugate} depends on this
section.

\needspace{10\baselineskip}
\subsection{The invariant limit register}

\begin{theorem}[The invariant limit register]\label{thm:register}
For $x,y \in X$ let $\lambda_k(x,y) = d(F^k x, F^k y)$,
$k \in \Z$. Then:
\begin{itemize}
\item[(i)] each ladder $k \mapsto \lambda_k(x,y)$ is nondecreasing
and nonnegative, so the limit
$\lambda^-(x,y) := \lim_{k \to -\infty} \lambda_k(x,y)$ exists;
\item[(ii)] $\lambda^-$ is a quasi-pseudo-metric on $X$ with
$\lambda^- \le d$;
\item[(iii)] $F$ is a bijective $\lambda^-$-\emph{isometry}:
$\lambda^-(Fx, Fy) = \lambda^-(x,y)$;
\item[(iv)] if $(X,d)$ is precompact, respectively hereditarily
precompact, so is $(X, \lambda^-)$.
\end{itemize}
Moreover, the \emph{defect} $D_0 := d - \lambda^-$ is nonnegative,
and $D_0(Fx,Fy) > 0$ whenever $d(Fx,Fy) > d(x,y)$.
\end{theorem}

\begin{proof}
(i) By the ladder relation of Section~\ref{sec:prelim}, the ladder
is nondecreasing, and it is bounded below by $0$, so it has a limit
as $k \to -\infty$.
(ii) We have $\lambda^-(x,x) = 0$, and letting $k \to -\infty$ in
$\lambda_k(x,z) \le \lambda_k(x,y) + \lambda_k(y,z)$ gives the
triangle inequality for $\lambda^-$. Since the ladders are
nondecreasing, $\lambda^- = \inf_{k \le 0} \lambda_k \le \lambda_0 = d$.
(iii) $\lambda^-(Fx,Fy) = \lim_k \lambda_{k+1}(x,y) =
\lambda^-(x,y)$.
(iv) Since $\lambda^- \le d$, every $d$-ball is contained in the
$\lambda^-$-ball with the same center and radius, so in every
subspace a finite $d$-net is a $\lambda^-$-net. Likewise, a sequence
that is left K-Cauchy for $d$ is left K-Cauchy for $\lambda^-$.

For the last assertion, suppose that $d(Fx,Fy) > d(x,y)$, that is,
$\lambda_1(x,y) > \lambda_0(x,y)$. By (iii) and (ii),
$D_0(Fx,Fy) = \lambda_1(x,y) - \lambda^-(x,y) \ge
\lambda_1(x,y) - \lambda_0(x,y) > 0$.
\end{proof}

Since $F$ is an isometry for $\lambda^-$, we call $\lambda^-$ the
\emph{isometric shadow} of $F$, or simply its \emph{register}. In
the toll spiral, the register is the base part and the defect is the
toll. Indeed, by Lemma~\ref{lem:tollmono}, $\tau(n+t,m+t) = 0$ for
all sufficiently negative $t$, so \eqref{eq:equiv} gives
$\lambda^-(n,m) = \rho(\varphi(n),\varphi(m))$, and hence
$D_0 = \tau$.

\subsection{The covering inequality}

\begin{proposition}[Recurrence fees cover the defect]
\label{prop:covering}
Let $x, y \in X$, and for $m \ge 1$ put
$r_m = d(x, F^{-m}x)$ and $s_m = d(F^{-m}y, y)$. Then
\[
r_m + s_m \;\ge\; \lambda_0(x,y) - \lambda_{-m}(x,y),
\qquad\text{hence}\qquad
\liminf_{m \to \infty}\, (r_m + s_m) \;\ge\; D_0(x,y).
\]
\end{proposition}

\begin{proof}
By the triangle inequality,
\[
d(x,y) \le d(x, F^{-m}x) + d(F^{-m}x, F^{-m}y) + d(F^{-m}y, y)
= r_m + \lambda_{-m}(x,y) + s_m .
\]
The second inequality follows by letting $m \to \infty$, since
$\lambda_{-m}(x,y) \to \lambda^-(x,y)$.
\end{proof}

In particular, if $D_0(x,y) > 0$, then $r_m$ and $s_m$ do not both
tend to $0$ along any subsequence. In the toll spiral, if $x \ge 0$,
then each pair $(x, x-m)$ with $m \ge 1$ is a descent with recent
source, so $r_m = d(x, x-m) \ge c$.

\subsection{The monochromy theorem}

\needspace{10\baselineskip}
A pair $(a,b)$ is \emph{strict} if $D_0(a,b) > 0$. Fix a strict
pair $(a,b)$ and put $\eta := D_0(a,b)$. For $\Delta \ge 1$, define
the
\emph{backward fees} and the \emph{self-defects}
\[
r(\Delta) = d(a, F^{-\Delta}a), \qquad
s(\Delta) = d(F^{-\Delta}b, b),
\]
\[
\delta_a(\Delta) = r(\Delta) - \lambda^-(a, F^{-\Delta}a) \ge 0,
\qquad
\delta_b(\Delta) = s(\Delta) - \lambda^-(F^{-\Delta}b, b) \ge 0.
\]
We call a set of displacements \emph{chain-closed} if it is the set
of differences $t_j - t_i$, $i < j$, of some increasing sequence
$t_1 < t_2 < \cdots$ in $\N$. The sum of two adjacent differences $t_j - t_i$ and
$t_k - t_j$, with $i < j < k$, is again in the set.

\begin{lemma}[Subadditivity of the fees]\label{lem:subadd}
$r(\Delta+\Delta') \le r(\Delta) + r(\Delta')$ and
$s(\Delta+\Delta') \le s(\Delta) + s(\Delta')$ for all
$\Delta, \Delta' \ge 1$.
\end{lemma}

\begin{proof}
By the triangle inequality,
\[
d(a, F^{-\Delta-\Delta'}a) \;\le\; d(a, F^{-\Delta}a) +
d(F^{-\Delta}a, F^{-\Delta-\Delta'}a).
\]
The second term equals $\lambda_{-\Delta}(a, F^{-\Delta'}a)$, which
is at most $\lambda_0(a, F^{-\Delta'}a) = r(\Delta')$ by the ladder
relation. Similarly, with $F^{-\Delta'}b$ as the intermediate point,
\[
d(F^{-\Delta-\Delta'}b, b) \;\le\;
\lambda_{-\Delta'}(F^{-\Delta}b, b) + s(\Delta') \;\le\;
s(\Delta) + s(\Delta'). \qedhere
\]
\end{proof}

\begin{lemma}[Large displacements]\label{lem:floor}
For every $\delta > 0$ there is an $M$ such that
$r(\Delta) + s(\Delta) \ge \eta - \delta$ for all $\Delta \ge M$.
\end{lemma}

\begin{proof}
By Proposition~\ref{prop:covering},
$r(\Delta) + s(\Delta) \ge d(a,b) - \lambda_{-\Delta}(a,b)$. As
$\Delta \to \infty$, the right-hand side increases to
$d(a,b) - \lambda^-(a,b) = \eta$.
\end{proof}

\begin{lemma}[Common recurrence chain]\label{lem:chain}
Suppose that $(X,d)$ is hereditarily precompact. Then for every
$\varepsilon > 0$ there is a sequence $t_1 < t_2 < \cdots$ in $\N$
such that $\lambda^-(a, F^{-\Delta}a) < \varepsilon$ and
$\lambda^-(F^{-\Delta}b, b) < \varepsilon$ for every difference
$\Delta = t_j - t_i$ with $i < j$.
\end{lemma}

\begin{proof}
By hereditary precompactness and \cite[Theorem~3]{KMRV1993}, every
sequence in $X$ has a left K-Cauchy subsequence; this characterization
needs no separation axiom. We apply this to the
sequence $(F^{-t}a)_{t \in \N}$, and then to $(F^{t}b)$ along the
subsequence obtained. This gives $t_1 < t_2 < \cdots$ such that
$(F^{-t_i}a)_i$ and $(F^{t_i}b)_i$ are both left K-Cauchy for $d$.
Since $\lambda^- \le d$, they are also left K-Cauchy for
$\lambda^-$. After discarding finitely many terms,
$\lambda^-(F^{-t_i}a, F^{-t_j}a) < \varepsilon$ and
$\lambda^-(F^{t_i}b, F^{t_j}b) < \varepsilon$ for all $i < j$. Since
$F$ is a $\lambda^-$-isometry (Theorem~\ref{thm:register}(iii)),
applying $F^{t_i}$ to the first pair gives
$\lambda^-(a, F^{-\Delta}a) < \varepsilon$, and applying $F^{-t_j}$
to the second gives $\lambda^-(F^{-\Delta}b, b) < \varepsilon$, where
$\Delta = t_j - t_i$.
\end{proof}

\begin{samepage}
\begin{theorem}[Monochromy: bijective strictness cannot
alternate]\label{thm:monochromy}
Let $(X,d)$ be hereditarily precompact, $F$ bijective
non-contractive, and $(a,b)$ strict with $\eta = D_0(a,b) > 0$.
For every $\varepsilon \in (0, \eta/6]$ there exist an infinite
chain-closed set $\mathcal{D}$ of displacements and an endpoint, say
$a$, such that
\[
\delta_a(\Delta) \;>\; \frac{\eta - \varepsilon}{2} -
\varepsilon \;\ge\; \frac{\eta}{4}
\qquad \text{for every } \Delta \in \mathcal{D}.
\]
Moreover, $\lambda^-(a, F^{-\Delta}a) < \varepsilon$ for every
$\Delta \in \mathcal{D}$. If the endpoint is $b$, the same holds
with $\delta_b(\Delta)$ and $\lambda^-(F^{-\Delta}b, b)$ in place of
$\delta_a(\Delta)$ and $\lambda^-(a, F^{-\Delta}a)$.
\end{theorem}
\end{samepage}

\begin{proof}
Take the sequence of Lemma~\ref{lem:chain} for this $\varepsilon$,
and let $M$ be as in Lemma~\ref{lem:floor} with $\delta =
\varepsilon$. Passing to a subsequence, we may assume that
consecutive terms differ by at least $M$. Then every difference
$\Delta = t_j - t_i$ with $i < j$ satisfies $\Delta \ge M$, so
$r(\Delta) + s(\Delta) \ge \eta - \varepsilon$, and the conclusion
of Lemma~\ref{lem:chain} still holds. Color a pair $\{i,j\}$ with $i < j$ by $\mathsf{a}$ if
$r(t_j - t_i) \ge (\eta-\varepsilon)/2$, and by $\mathsf{b}$
otherwise. Since $r + s \ge \eta - \varepsilon$, every pair of color
$\mathsf{b}$ satisfies $s(t_j - t_i) \ge (\eta-\varepsilon)/2$. By
the infinite Ramsey theorem for pairs, some infinite set
$H \subseteq \N$ has all its pairs of the same color. We may assume
that this color is $\mathsf{a}$; for $\mathsf{b}$, the same argument
applies with $s$, $\delta_b$ and $\lambda^-(F^{-\Delta}b, b)$. Let $\mathcal{D}$ be the set of differences
$t_j - t_i$ with $i < j$ in $H$. It is chain-closed, since it is the
set of differences of the sequence $(t_h)_{h \in H}$. Each
$\Delta \in \mathcal{D}$ is also a difference of the sequence of
Lemma~\ref{lem:chain}, so $\lambda^-(a, F^{-\Delta}a) < \varepsilon$,
while $r(\Delta) \ge (\eta-\varepsilon)/2$. Hence
\[
\delta_a(\Delta) = r(\Delta) - \lambda^-(a, F^{-\Delta}a)
> \frac{\eta-\varepsilon}{2} - \varepsilon \ge \frac{\eta}{4},
\]
where the last inequality uses $\varepsilon \le \eta/6$.
\end{proof}

We keep the theorem in its version-1 form, the form to which
\cite{DuzerS} refers: either endpoint can carry the defect, with the
bound $\eta/4$. Corollary~3.3 of \cite{DuzerS} proves a sharper form:
the endpoint can always be taken to be $a$, and the bound
$\eta/4$ can be replaced by any number less than $\eta$. The proof
above also gives it, without the Ramsey theorem. In the proof of
Lemma~\ref{lem:chain}, the sequence $(F^{t_i}b)_i$ is left K-Cauchy
for $d$, so, after discarding finitely many terms,
$d(F^{t_i}b, F^{t_j}b) < \varepsilon$ for all $i < j$. Since
$F^{-t_j}$ is non-expansive, it follows that $s(\Delta) < \varepsilon$
for every difference $\Delta = t_j - t_i$. After thinning the chain
so that all its differences are at least $M$, Lemma~\ref{lem:floor}
gives $r(\Delta) > \eta - 2\varepsilon$, and hence
$\delta_a(\Delta) > \eta - 3\varepsilon$.

\begin{remark}[No co-avoidance]\label{rem:noalt}
For a strict pair, $r$ cannot be small at two displacements while
$s$ is small at their sum, once that sum is large. Indeed, let
$0 < \varepsilon' < \eta/3$, and suppose
that $r(\Delta_1) \le \varepsilon'$, $r(\Delta_2) \le \varepsilon'$
and $s(\Delta_1 + \Delta_2) \le \varepsilon'$. By
Lemma~\ref{lem:subadd}, $r(\Delta_1 + \Delta_2) \le 2\varepsilon'$,
so $r(\Delta_1+\Delta_2) + s(\Delta_1+\Delta_2) \le 3\varepsilon'$.
Choose $\delta > 0$ with $3\varepsilon' + \delta < \eta$, and let
$M$ be as in Lemma~\ref{lem:floor}. If $\Delta_1 + \Delta_2 \ge M$,
the same sum is at least $\eta - \delta > 3\varepsilon'$, a
contradiction.
\end{remark}

\begin{remark}[The design space, informally]\label{rem:design}
This remark is heuristic, and nothing in the paper depends on it.
Before finding the toll spiral, we tried designs that added to the
distance a surcharge graded by a quantity monotone in time, such as
the index, a height, or an arrival depth. In each of these designs,
the surcharge tended to $0$ as the pair was moved into the past, and
none of them gave a counterexample. Some of them kept the
round trip $d(x,y) + d(y,x)$ constant, and
Proposition~\ref{prop:conserve} below shows that $F$ is then an
isometry. The other designs failed in one of two ways: the
register did not separate points, or the triangle inequality failed
for routes through the deep past.

In the toll spiral, the toll does not vary: it is the constant $c$,
and it is charged only on descents with a recent source. We call the
set of these descents the \emph{toll zone}.
The base part is the isometric shadow, and by
Lemma~\ref{lem:tollmono} the shift $(n,m) \mapsto (n+1,m+1)$ maps
the toll zone into itself. No pair of deep integers lies in the
toll zone, so $d$ is symmetric on the deep integers. The ladders
increase although the toll is constant: for a descent $(n,m)$, the
pair $(n+t,m+t)$ lies outside the toll zone for $t < -n$ and inside
it for $t \ge -n$. In this way the toll spiral avoids these
obstacles. Question~\ref{q:forced} asks whether every
counterexample has this form.
\end{remark}

\begin{proposition}[Plasticity by conservation]
\label{prop:conserve}
Call a quasi-metric space $(X,d)$ \emph{conservative} if there
is a constant $M'$ with $d(x,y) + d(y,x) = M'$ for all
$x \ne y$. Every non-contractive self-map of a conservative
quasi-metric space is an isometry (bijectivity is not needed).
In particular, for any continuous density $w > 0$ on the
circle, the counterclockwise arc integral
$d_w(x,y) = \int_{x \curvearrowright y} w$ is a
conservative quasi-metric with $M' = \int w$, and no density
grading of the circle yields a counterexample.
\end{proposition}

\begin{proof}
Let $x \ne y$. By separation, $d(x,y) > 0$, so
$d(Fx,Fy) \ge d(x,y) > 0$ and $Fx \ne Fy$. Hence
$d(Fx,Fy) + d(Fy,Fx) = M' = d(x,y) + d(y,x)$. Neither term on the
left is smaller than the corresponding term on the right, so
$d(Fx,Fy) = d(x,y)$.

Now consider $d_w$. If $x \ne y$, the counterclockwise arc from
$x$ to $y$ is nondegenerate, so $d_w(x,y) > 0$ because $w$ is
positive and continuous. The counterclockwise arcs from $x$ to $y$
and from $y$ to $z$ together cover the counterclockwise arc from $x$
to $z$, which gives the triangle inequality. Finally, the two arcs
between distinct points $x$ and $y$ make up the whole circle, so
$d_w(x,y) + d_w(y,x) = \int w$.
\end{proof}

The space $(\R/\Z, d_w)$ is precompact. Given $\varepsilon > 0$,
cut the circle into finitely many arcs over each of which $w$ has
integral less than $\varepsilon$; their endpoints form an
$\varepsilon$-net. But the space is not hereditarily precompact.
Take a sequence $(x_k)$ that converges to a point while moving
clockwise towards it. For $k < l$, the counterclockwise arc from
$x_k$ to $x_l$ covers almost the whole circle, so $d_w(x_k,x_l)$
is close to $M'$ when $k$ is large. Hence no subsequence of $(x_k)$
is left K-Cauchy.

\section{Machine verification}\label{sec:mech}

We have formalized the main theorem (Theorem~\ref{thm:family}) and
its circle instance (Corollary~\ref{cor:circle}) in Lean~4 using
Mathlib, together with the results of
Section~\ref{sec:construction} on which they rest. Everything is in
one file, \texttt{TollSpiral.lean}, which contains no
\texttt{sorry}. The $p$-adic instance (Corollary~\ref{cor:padic}),
the necessity of (A2) noted after Theorem~\ref{thm:family}, and
Sections~\ref{sec:conjugate}, \ref{sec:anatomy}
and~\ref{sec:questions} are not formalized.

The formalization has two layers. The first works with any metric space $P$ and any injective
equivariant map $\varphi\colon \Z \to P$ with totally bounded range.
In the second, $\varphi$ is derived from an isometric
self-equivalence $\sigma$ of $P$ with an aperiodic point, so
equivariance \eqref{eq:equiv} is a theorem rather than a hypothesis.
Compactness of $P$ is not assumed in either layer, only the total
boundedness of the orbit, which is (A2). For the circle, (A2) follows
from compactness.

\begin{samepage}
The results of Section~\ref{sec:construction} correspond to the
following declarations.
\nopagebreak
\begin{center}
\begin{tabular}{ll}
Definition~\ref{def:toll} & \texttt{toll}, \texttt{d}\\
Lemma~\ref{lem:tollsub} & \texttt{toll\_triangle}\\
Lemma~\ref{lem:tollmono}, first claim & \texttt{toll\_mono}\\
Proposition~\ref{prop:qm} & \texttt{d\_self},
\texttt{d\_triangle}, \texttt{d\_separation}\\
Proposition~\ref{prop:noncontr} & \texttt{d\_noncontractive},
\texttt{d\_strict}\\
Bijectivity of $S$ & \texttt{shift\_bijective}\\
Lemma~\ref{lem:extract} &
\texttt{exists\_cauchy\_subseq\_of\_totallyBounded}\\
Theorem~\ref{thm:hered} & \texttt{hereditarilyPrecompact}\\
Theorem~\ref{thm:hered} (net form) &
\texttt{hereditarilyPrecompact\_net}\\
Boundedness (Theorem~\ref{thm:family}) & \texttt{d\_bounded}\\
Theorem~\ref{thm:family} & \texttt{counterexample\_family}\\
Corollary~\ref{cor:circle} & \texttt{counterexample},
\texttt{counterexample\_sqrt2}
\end{tabular}
\end{center}
\end{samepage}

Theorem~\ref{thm:hered} is mechanized in two forms.
\texttt{hereditarilyPrecompact} is the sequential statement, and
\texttt{hereditarilyPrecompact\_net} is the definition itself: for
every subset and every $\varepsilon > 0$, it gives a finite
$\varepsilon$-net of forward balls with centers in the subset. The
formal statement therefore covers the hypothesis of the question as
stated, and does not rely on the characterization of
\cite{KMRV1993}. The net form, boundedness and nonnegativity of $d$
are also part of the conclusions of \texttt{counterexample\_family}
and \texttt{counterexample}, so these two theorems state the
hypothesis of the question directly.

The formal proofs mostly follow the paper.
Lemma~\ref{lem:extract} is proved through the metric completion,
where the closure of the range is compact, instead of by the diagonal
argument given here. The formal version of Proposition~\ref{prop:noncontr} proves
non-contractivity, and proves strictness only for the pair
$(-1,-2)$, that is, $d(-1,-2) < d(0,-1)$. The formula
$d(S^{|n|} n, S^{|n|} m) = d(n,m) + c$ for $m < n < 0$ is proved only
on paper. For Corollary~\ref{cor:circle}, the circle is
\texttt{AddCircle}~$1$, and the file proves that the rotation by $\theta$ is an isometry,
that its orbit is $n \mapsto n\theta$, and that this orbit is
aperiodic when $\theta$ is irrational. The theorem \texttt{counterexample\_sqrt2}
treats the case $\theta = \sqrt{2}$, $c = 1$, and states
non-contractivity and strictness there.

For each of the following eleven theorems, a \texttt{\#guard\_msgs}
check on \texttt{\#print axioms} ensures at compile time that the
theorem depends on exactly \texttt{propext}, \texttt{Classical.choice}
and \texttt{Quot.sound}, the standard axioms of classical Lean:
\begin{center}\small
\texttt{toll\_triangle}, \texttt{d\_triangle},
\texttt{shift\_bijective}, \texttt{d\_separation},
\texttt{d\_noncontractive}, \texttt{d\_strict},\\
\texttt{hereditarilyPrecompact},
\texttt{hereditarilyPrecompact\_net},\\
\texttt{counterexample\_family}, \texttt{counterexample},
\texttt{counterexample\_sqrt2}.
\end{center}

\begin{samepage}
For fixed $c$ and $\varphi$, the main definitions in the file are as
follows.
\nopagebreak
\begin{small}
\begin{verbatim}
def toll (n m : ℤ) : ℝ := if m < n ∧ 0 ≤ n then c else 0

def d (n m : ℤ) : ℝ := dist (φ n) (φ m) + toll c n m

def LeftKCauchy (w : ℕ → ℤ) : Prop :=
  ∀ ε : ℝ, 0 < ε → ∃ N : ℕ, ∀ k l : ℕ,
    N ≤ k → k ≤ l → d c φ (w k) (w l) < ε

def orbitMap (σ : P ≃ᵢ P) (y₀ : P) (n : ℤ) : P := (σ ^ n) y₀
\end{verbatim}
\end{small}
\end{samepage}
\goodbreak

\begin{samepage}
The two main theorems are stated as follows.
\nopagebreak
\begin{small}
\begin{verbatim}
theorem counterexample_family (c : ℝ) (hc : 0 < c)
    (σ : P ≃ᵢ P) (y₀ : P)
    (haper : ∀ k : ℤ, k ≠ 0 → (σ ^ k) y₀ ≠ y₀)
    (hTB : TotallyBounded (Set.range (orbitMap σ y₀))) :
    Function.Bijective (fun n : ℤ => n + 1) ∧
    (∀ n : ℤ, d c (orbitMap σ y₀) n n = 0) ∧
    (∀ n m k : ℤ, d c (orbitMap σ y₀) n k ≤
      d c (orbitMap σ y₀) n m + d c (orbitMap σ y₀) m k) ∧
    (∀ n m : ℤ, d c (orbitMap σ y₀) n m = 0 → n = m) ∧
    (∀ n m : ℤ, d c (orbitMap σ y₀) n m ≤
      d c (orbitMap σ y₀) (n + 1) (m + 1)) ∧
    (d c (orbitMap σ y₀) (-1) (-2) <
      d c (orbitMap σ y₀) (-1 + 1) (-2 + 1)) ∧
    (∀ u : ℕ → ℤ, ∃ g : ℕ → ℕ, StrictMono g ∧
      LeftKCauchy c (orbitMap σ y₀) (u ∘ g)) ∧
    (∃ B : ℝ, ∀ n m : ℤ, d c (orbitMap σ y₀) n m ≤ B) ∧
    (∀ W : Set ℤ, ∀ ε : ℝ, 0 < ε → ∃ F : Finset ℤ,
      ↑F ⊆ W ∧ ∀ y ∈ W, ∃ a ∈ F, d c (orbitMap σ y₀) a y < ε) ∧
    (∀ n m : ℤ, 0 ≤ d c (orbitMap σ y₀) n m)

theorem hereditarilyPrecompact_net
    (hTB : TotallyBounded (Set.range φ))
    (W : Set ℤ) {ε : ℝ} (hε : 0 < ε) :
    ∃ F : Finset ℤ, ↑F ⊆ W ∧ ∀ y ∈ W, ∃ a ∈ F, d c φ a y < ε
\end{verbatim}
\end{small}
\end{samepage}

\section{Questions}\label{sec:questions}

The negative answer raises three questions. We state them below,
together with three boundary results on the first.

\begin{question}[Compactness]\label{q:compact}
Does some notion of compactness for quasi-metric spaces (there are
several inequivalent ones; see \cite{Cobzas2013}) force bijective
non-contractive self-maps to be isometries? This is Zavarzina's
question with precompactness replaced by compactness.
\end{question}

\needspace{8\baselineskip}
\begin{samepage}
The following three results delimit the question.
\nopagebreak

\begin{proposition}[The symmetrized case is positive]
\label{prop:symcompact}
Let $(X,d)$ be a quasi-metric space whose symmetrization
$d_s = d + d^{-1}$ is totally bounded (for instance, compact). Then
every non-contractive self-map of $X$, bijective or not, is an
isometry.
\end{proposition}
\end{samepage}

\begin{proof}
Total boundedness of a metric passes to subspaces, and since
$d \le d_s$ and $d^{-1} \le d_s$, every $d_s$-net is both a $d$-net
and a $d^{-1}$-net. Hence $(X,d)$ and $(X,d^{-1})$ are both
hereditarily precompact. By Theorem~3 of \cite{Zav2023}, a map that
expands some pair also contracts another, so a non-contractive
self-map expands no pair and is an isometry.
\end{proof}

In the next two results we use the forward-ball topology, in which a
set is open if it contains a forward ball around each of its points.
A space is \emph{left K-sequentially complete} if every left K-Cauchy
sequence $(x_k)$ has a limit, that is, a point $p$ with
$d(p, x_k) \to 0$.

Proposition~\ref{prop:deepcopy} adds a copy of the circle below the
toll spiral of Corollary~\ref{cor:circle}. Let
$X = \Z \sqcup (\R/\Z)$, and write $y^{*}$ for the copy of
$y \in \R/\Z$. Every copy is placed below every integer, and the
copies are pairwise incomparable. The toll is defined as before:
$\tau(a,b) = c$ if $a$ is a recent integer and $b$ lies below $a$,
and $\tau(a,b) = 0$ otherwise. In particular, the order among the
copies plays no role. Put $\psi(n) =
n\theta$ and $\psi(y^{*}) = y$, let
$d(a,b) = \rho(\psi a, \psi b) + \tau(a,b)$, and let $S(n) = n+1$
and $S(y^{*}) = (y+\theta)^{*}$.

\begin{proposition}[A compact, complete, non-separated
counterexample]\label{prop:deepcopy}
Let $X$, $d$ and $S$ be as above. Then $(X,d)$ is a bounded,
hereditarily precompact quasi-pseudo-metric space, compact in the
forward-ball topology and left K-sequentially complete, and $S$ is a
bijective non-contractive map that is not an isometry. However,
$(X,d)$ is not $T_0$.
\end{proposition}

\begin{proof}
We first check that the toll is subadditive. Let $a, k, m \in X$
with $\tau(a,k) = c$, so that $a$ is a recent integer and $k$ lies
below $a$. If $m$ lies below $a$, then $\tau(a,m) = c$. Otherwise
$m$ is an integer with $m \ge a$, so $m$ is recent and $k$ lies
below $m$, and $\tau(m,k) = c$. The proof of Lemma~\ref{lem:tollsub}
used that the order is total; here it is enough that all copies lie
below all integers. The quasi-pseudo-metric axioms then follow as in
Proposition~\ref{prop:qm}.

The map $S$ is a bijection, and it preserves the base part, since
$\psi(Sa) = \psi a + \theta$ for every $a \in X$. It also preserves
recency and the order, so the toll does not decrease under $S$, as
in Lemma~\ref{lem:tollmono}. Hence $S$ is non-contractive, and it
increases the distance of the pair $(-1,-2)$ by $c$.

The direct proof of Theorem~\ref{thm:hered} gives hereditary
precompactness. In that proof, the deep part is now formed by the
copies and the negative integers. No toll is charged from these
points, so $d(a,b) = \rho(\psi a, \psi b)$ on the deep part. Since
$\psi$ takes values in the compact circle, the net argument goes
through as before, as it does not use injectivity of the base map.
The recent part is the same as before.

Since no toll is charged from a copy, the forward ball of radius
$\varepsilon$ around $z^{*}$ is $\psi^{-1}(B_\rho(z,\varepsilon))$.
Let $\mathcal{U}$ be an open cover of $X$. For each $z$ in the
circle, some member of $\mathcal{U}$ contains $z^{*}$, and therefore
contains $\psi^{-1}(B_\rho(z,\varepsilon_z))$ for some
$\varepsilon_z > 0$. Finitely many of the balls
$B_\rho(z,\varepsilon_z)$ cover the circle, and the corresponding
members of $\mathcal{U}$ cover $X$.

For completeness, let $(x_k)$ be left K-Cauchy. Then $(\psi x_k)$ is
$\rho$-Cauchy, with a limit $z$ in the circle, and
$d(z^{*}, x_k) = \rho(z, \psi x_k) \to 0$.

Finally, for a negative integer $j$, the points $j$ and
$(j\theta)^{*}$ are distinct and at distance $0$ in both directions,
so $(X,d)$ is not $T_0$.
\end{proof}

Proposition~\ref{prop:incomplete} concerns spaces obtained from the
toll spiral by adding points. Let $X'$ be a set containing $\Z$,
with a strict order that extends the order of $\Z$ and a map
$\psi\colon X' \to Y$ that extends $\varphi$. Choose a set of
\emph{recent} points in $X'$ whose integers are exactly those
$n \ge 0$, and define $d'(a,b) = \rho(\psi a, \psi b) + \tau(a,b)$,
where $\tau(a,b) = c$ if $a$ is recent and $b$ lies below $a$, and
$\tau(a,b) = 0$ otherwise. If $d'$ is a quasi-metric, we call
$(X',d')$ an \emph{extension} of the toll spiral. This is the case,
for instance, when $\tau$ is subadditive on $X'$ and $\psi$ is
injective. An extension keeps the single charge $c$ and the rule for
charging it.

\begin{proposition}[Separated toll extensions are incomplete
and non-compact]\label{prop:incomplete}
In the setting of Theorem~\ref{thm:family}, neither $(\Z, d)$ nor
any extension of it is left K-sequentially
complete or countably compact in the forward-ball topology. In
particular, none of these spaces is compact in the forward-ball
topology.
\end{proposition}

\begin{proof}
Let $X$ be either $\Z$, with the distance $d$ and $\psi = \varphi$,
or an extension, whose distance we also denote by
$d$. We first show that $\psi$ is injective. Let $a \ne b$ with
$\psi a = \psi b$. Then $d(a,b) = \tau(a,b)$ and
$d(b,a) = \tau(b,a)$. A toll is charged on a pair only when its
second point lies below its first, so at least one of these two
distances is $0$, which contradicts separation.

Since the orbit is totally bounded, there are $t_i < t_j$ with
$\rho(\varphi(-t_i), \varphi(-t_j))$ arbitrarily small, and by
equivariance this distance equals
$\rho(\varphi(0), \varphi(-(t_j - t_i)))$. This gives negative
integers $n_k$ with $\varphi(n_k) \to \varphi(0)$. Since the $n_k$
are deep, $d(n_k, n_l) = \rho(\varphi(n_k), \varphi(n_l))$ for all
$k$ and $l$, so $(n_k)$ is left K-Cauchy.

Suppose that $p \in X$ is a limit of $(n_k)$. Then
$\rho(\psi p, \varphi(n_k)) \to 0$, so
$\psi p = \varphi(0) = \psi(0)$, and $p = 0$ because $\psi$ is
injective. But each $(0, n_k)$ is a descent with recent source $0$,
so $d(0, n_k) \ge c$ for every $k$, a contradiction. Hence $X$ is
not left K-sequentially complete.

The same argument excludes cluster points. By the triangle
inequality, forward balls are open in the forward-ball topology. If
$p$ were a cluster point of $(n_k)$, every forward ball around $p$
would contain $n_k$ for infinitely many $k$, and as before
$\psi p = \varphi(0)$ and $p = 0$. But the ball of radius $c$ around
$0$ contains no $n_k$. So $(n_k)$ has no cluster point, and $X$ is
not countably compact.
\end{proof}

Thus Question~\ref{q:compact} has a positive answer when the
symmetrization is totally bounded (Proposition~\ref{prop:symcompact})
and a negative answer without separation
(Proposition~\ref{prop:deepcopy}), and no extension of the toll
spiral is compact in the forward-ball topology
(Proposition~\ref{prop:incomplete}). The open case is that of a
separated space that is compact in one of its two topologies. One approach uses the sharper form of
Theorem~\ref{thm:monochromy} given after it: along the chain of
Lemma~\ref{lem:chain}, the backward iterates of the endpoint carrying
the defect are left K-Cauchy. In a complete space they have a limit.
Using such limits, the author later showed in \cite{DuzerC} that every
non-contractive bijection of a complete, separated, hereditarily
precompact space is an isometry; see the note below.

\begin{question}[The non-separated setting]\label{q:nonsep}
Without axiom (c), a quasi-pseudo-metric space can have
\emph{free moves}, that is, distinct points at distance zero in one
direction. The construction of Section~\ref{sec:construction}
satisfies separation, as the question of \cite{Zav2023} requires.
Whether the expand-contract dichotomy survives, fails more severely,
or becomes trivial in the non-separated (Lawvere) setting does not
seem to have been studied. Proposition~\ref{prop:deepcopy} is a
first example: without separation, the dichotomy fails even for
compact, left K-sequentially complete spaces. A finer question lies
in the gap between Propositions~\ref{prop:deepcopy} and~\ref{prop:incomplete}.
Axiom (c) requires $d(x,y) > 0$ \emph{and} $d(y,x) > 0$ for distinct
points, while $T_0$ requires only one of the two. The example of
Proposition~\ref{prop:deepcopy} is not $T_0$, since for each
negative integer $j$ the points $j$ and $(j\theta)^{*}$ are at
distance $0$ in both directions. For $j \ge 0$, the distance from
$(j\theta)^{*}$ to $j$ is still $0$, but the distance back is $c$,
so the example is not $R_0$ either (in an $R_0$ space,
$d(x,y) = 0$ implies $d(y,x) = 0$). The separation argument of
Proposition~\ref{prop:incomplete} uses (c) in full, not just
$T_0$. Is there a $T_0$ example that is compact and left
K-sequentially complete?
\end{question}

\begin{question}[Forced form]\label{q:forced}
Is every bijective
non-contractive map that is not an isometry, on a hereditarily
precompact quasi-metric space, conjugate in an appropriate sense to
a tolled extension (a fixed toll zone over an isometric shadow),
possibly with infinitely many zones? The results of
Section~\ref{sec:anatomy}, in particular the decomposition
$d = \lambda^- + D_0$ of Theorem~\ref{thm:register} and the
monochromy theorem, point toward a positive answer but do not prove
it. The failed designs
described in Remark~\ref{rem:design} are evidence rather than proof,
except in the conservation case (Proposition~\ref{prop:conserve}).
\end{question}

\paragraph{Note added in version 2.}
Question~\ref{q:compact} has since been answered by the author
\cite{DuzerC}: for quasi-metric spaces, compactness of either
topology forces plasticity, and hereditary precompactness together
with sequential compactness of the conjugate topology forces every
non-contractive self-map to be an isometry. The last question in
Question~\ref{q:nonsep}, whether there is a compact, left
K-sequentially complete $T_0$ example, has an affirmative answer,
also in \cite[Proposition~6.1]{DuzerC}. Version~1 of this paper
was posted on Zenodo on 17~September 2026
(\href{https://doi.org/10.5281/zenodo.22801325}{doi:\nolinkurl{10.5281/zenodo.22801325}})
and submitted the same day to arXiv
(\href{https://arxiv.org/abs/2610.13973}{arXiv:2610.13973},
version~1 timestamped 17~September 2026, announced in
October 2026).
Agyingi \cite{Agyingi2026}, in a preprint submitted to arXiv on
7~October 2026, has since studied non-expansive bijections of
$T_0$-quasi-metric spaces, a setting that bears on
Question~\ref{q:nonsep}; \cite{DuzerC} compares his results with its
own.

Question~\ref{q:forced} is answered by the author in \cite{DuzerS}:
negatively for tolled extensions that charge only descents, as in
Remark~\ref{rem:design}, and positively in structural form. In a
tolled extension of that kind, ascents carry no defect, and the
distance from a point to its image is constant along its orbit. The
ascent spiral of \cite{DuzerS} (Example~6.7 there) has neither
property: the first failure excludes conjugacies that preserve the
metric, and the second, since its distances from $n$ to $n + 1$ are
unbounded, excludes bi-Lipschitz ones (Remark~6.3 there). On each
infinite orbit, a map is a tolled extension with possibly infinitely
many zones of descents over its register exactly when its ascents
carry no defect (Proposition~6.10 there). In structural form, every
non-contractive bijection of a hereditarily precompact quasi-metric
space is a toll over its rigid shadow (Corollary~1.3 there). The
rigid shadow is the isometric shadow of Section~\ref{sec:anatomy},
which on each orbit completes to a rotation of a compact monothetic
group.

Graded tolls also give counterexamples. The height spiral of
\cite{DuzerS} (Example~6.1 there) is graded by height. The space of
\cite[Proposition~6.1]{DuzerC} contains the integers as an invariant
layer, on which the map defined there is the shift, and its toll is
graded by the index: a descent from $n$ is charged $g(n)$, where $g$
is a positive nondecreasing function tending to $0$ at the deep end.
So the designs of Remark~\ref{rem:design} did not fail merely because
they were graded, and their failure is weaker evidence for a positive
answer to Question~\ref{q:forced}.

\section*{Data availability}

The Lean file \texttt{TollSpiral.lean}, together with the files
needed to build it, is available at Zenodo
(\href{https://doi.org/10.5281/zenodo.24015027}{doi:\nolinkurl{10.5281/zenodo.24015027}}).
The same files accompany the arXiv version
(\href{https://arxiv.org/abs/2610.13973}{arXiv:2610.13973}) as
ancillary files.
It builds with Lean~4.32.0 and Mathlib commit
\texttt{81a5d257c8e410db227a6665ed08f64fea08e997}, without \texttt{sorry} or warnings. The SHA-256
checksum of the file is
\begin{center}\small
\texttt{9c02c6903221ecc843ed15d8456ef939f9eb100034e03fc1710acffa45397dec}.
\end{center}

\section*{Acknowledgements}

The author developed the mathematics in a draft in French and
used AI to translate it into a first English version. The author
corrected this version wherever it departed from the French and
rewrote it, keeping some of its standard phrasing. The author
then used AI to check the grammar and reviewed the changes.

\needspace{12\baselineskip}


\begin{thebibliography}{99}\footnotesize

\bibitem{Agyingi2026}
Agyingi, C.A.:
Nonexpansive bijections in $T_0$-quasi-metric spaces and their
$q$-hyperconvex hulls.
Preprint (2026).
\href{https://arxiv.org/abs/2610.10077v1}{arXiv:2610.10077v1}.

\bibitem{AKZ2019}
Angosto, C., Kadets, V., Zavarzina, O.:
Non-expansive bijections, uniformities and polyhedral faces.
Journal of Mathematical Analysis and Applications 471(1--2),
38--52 (2019). \href{https://doi.org/10.1016/j.jmaa.2018.10.058}{doi:\nolinkurl{10.1016/j.jmaa.2018.10.058}}.

\bibitem{BMZ2026}
Banakh, T., Mazurenko, O., Zavarzina, O.:
Plastic metric spaces and groups.
Topology and its Applications 384 (2026), 109802.
\href{https://doi.org/10.1016/j.topol.2026.109802}{doi:\nolinkurl{10.1016/j.topol.2026.109802}};
\href{https://arxiv.org/abs/2510.10537}{arXiv:2510.10537}.

\bibitem{CKOW2016}
Cascales, B., Kadets, V., Orihuela, J., Wingler, E.J.:
Plasticity of the unit ball of a strictly convex Banach space.
Revista de la Real Academia de Ciencias Exactas, F\'isicas y
Naturales, Serie A, Matem\'aticas 110(2), 723--727 (2016).
\href{https://doi.org/10.1007/s13398-015-0261-3}{doi:\nolinkurl{10.1007/s13398-015-0261-3}}.

\bibitem{Cobzas2013}
Cobza\c{s}, \c{S}.:
Functional Analysis in Asymmetric Normed Spaces.
Frontiers in Mathematics, Birkh\"auser (2013).
\href{https://doi.org/10.1007/978-3-0348-0478-3}{doi:\nolinkurl{10.1007/978-3-0348-0478-3}}.

\bibitem{DuzerC}
Duzer, T.:
Compact quasi-metric spaces are expand-contract plastic.
Zenodo (2026).
\href{https://doi.org/10.5281/zenodo.23265517}{doi:\nolinkurl{10.5281/zenodo.23265517}}
(concept DOI, resolving to the latest version; the results cited
here are at version~3,
\href{https://doi.org/10.5281/zenodo.24015052}{doi:\nolinkurl{10.5281/zenodo.24015052}});
\href{https://arxiv.org/abs/2610.14301}{arXiv:2610.14301}.

\bibitem{DuzerS}
Duzer, T.:
The structure of a non-isometric non-contractive bijection of a
hereditarily precompact quasi-metric space.
Zenodo (2026).
\href{https://doi.org/10.5281/zenodo.23544811}{doi:\nolinkurl{10.5281/zenodo.23544811}}
(concept DOI, resolving to the latest version; the numbered results
cited here are at version~2,
\href{https://doi.org/10.5281/zenodo.24015069}{doi:\nolinkurl{10.5281/zenodo.24015069}});
\href{https://arxiv.org/abs/2611.03681}{arXiv:2611.03681}.

\bibitem{FH1936}
Freudenthal, H., Hurewicz, W.:
Dehnungen, Verk\"urzungen, Isometrien.
Fundamenta Mathematicae 26, 120--122 (1936).
\href{https://doi.org/10.4064/fm-26-1-120-122}{doi:\nolinkurl{10.4064/fm-26-1-120-122}}.

\bibitem{KZ2018}
Kadets, V., Zavarzina, O.:
Non-expansive bijections to the unit ball of the $\ell_1$-sum of
strictly convex Banach spaces.
Bulletin of the Australian Mathematical Society 97(2), 285--292
(2018). \href{https://doi.org/10.1017/S0004972717001150}{doi:\nolinkurl{10.1017/S0004972717001150}}.

\bibitem{KZ2024}
Kadets, V., Zavarzina, O.:
Plastic pairs of metric spaces.
Journal of Mathematical Analysis and Applications 529, 127435
(2024).
\href{https://arxiv.org/abs/2301.06589}{arXiv:2301.06589};
\href{https://doi.org/10.1016/j.jmaa.2023.127435}{doi:\nolinkurl{10.1016/j.jmaa.2023.127435}}.

\bibitem{KMRV1993}
K\"unzi, H.-P.A., Mr\v{s}evi\'c, M., Reilly, I.L.,
Vamanamurthy, M.K.:
Convergence, precompactness and symmetry in quasi-uniform spaces.
Mathematica Japonica 38, 239--253 (1993).

\bibitem{LZ2024}
Langemann, D., Zavarzina, O.:
Expand-contract plasticity on the real line.
Frontiers in Applied Mathematics and Statistics 10, 1387012
(2024). \href{https://doi.org/10.3389/fams.2024.1387012}{doi:\nolinkurl{10.3389/fams.2024.1387012}}.

\bibitem{Lawvere1973}
Lawvere, F.W.:
Metric spaces, generalized logic, and closed categories.
Rendiconti del Seminario Matematico e Fisico di Milano 43,
135--166 (1973). Reprinted in:
\href{http://www.tac.mta.ca/tac/reprints/articles/1/tr1abs.html}{Reprints
in Theory and Applications of Categories 1 (2002), 1--37}.

\bibitem{NPW2006}
Naimpally, S.A., Piotrowski, Z., Wingler, E.J.:
Plasticity in metric spaces.
Journal of Mathematical Analysis and Applications 313, 38--48
(2006). \href{https://doi.org/10.1016/j.jmaa.2005.04.070}{doi:\nolinkurl{10.1016/j.jmaa.2005.04.070}}.

\bibitem{Zav2018}
Zavarzina, O.:
Non-expansive bijections between unit balls of Banach spaces.
Annals of Functional Analysis 9(2), 271--281 (2018).
\href{https://doi.org/10.1215/20088752-2017-0050}{doi:\nolinkurl{10.1215/20088752-2017-0050}}.

\bibitem{Zav2023}
Zavarzina, O.:
On expand-contract plasticity in quasi-metric spaces.
Carpathian Mathematical Publications 15(2), 524--528 (2023).
\href{https://doi.org/10.15330/cmp.15.2.524-528}{doi:\nolinkurl{10.15330/cmp.15.2.524-528}};
\href{https://arxiv.org/abs/2207.11513}{arXiv:2207.11513};
\href{https://zbmath.org/?q=an%3A1539.54011}{Zbl 1539.54011}.

\end{thebibliography}
\end{document}